\documentclass[12pt,twoside]{amsart}
\usepackage{amsfonts,amssymb,amsthm,amsmath,amsxtra,amscd,verbatim,eucal}

\usepackage[all]{xy}
\usepackage[dvips]{graphics}

\title
{Geometric properties of jet schemes}
\author{Shihoko Ishii} 

\address{Department of Mathematics, Tokyo Institute of
Technology, Oh-Okayama, Meguro, Tokyo, Japan
\newline
e-mail : shihoko@@math.titech.ac.jp}

\newcommand{\bC}{{\Bbb C}}

\newcommand{\bQ}{{\Bbb Q}}

\newcommand{\bN}{{\Bbb N}}
\newcommand{\bA}{{\Bbb A}}

\newcommand{\R}{{\mathcal R}}

\newcommand{\Proj}{\operatorname{Proj}}
\newcommand{\Hom}{\operatorname{Hom}}

\newcommand{\stm}{{\spec k[t]/(t^{m+1})}}
\newcommand{\sTm}{{\spec K[t]/(t^{m+1})}}

\newcommand{\tm}{{k[t]/(t^{m+1})}}

\newcommand{\supp}{\operatorname{Supp}}

\let \cedilla =\c
\renewcommand{\c}[0]{{\mathbb C}}  

\renewcommand{\o}[0]{{\mathcal O}} 

\newcommand{\spec}[0]{\operatorname{Spec}}
\newcommand{\sing}[0]{\operatorname{Sing}}
\newcommand{\oa}{\overline{A}}
\newcommand{\ob}{\overline{B}}

\def\to {\longrightarrow}

\newtheorem{thm}{Theorem}[section]

\newtheorem{lem}[thm]{Lemma}
\newtheorem{cor}[thm]{Corollary}
\newtheorem{prop}[thm]{Proposition}

\theoremstyle{definition}
\newtheorem{defn}[thm]{Definition}

\newtheorem{say}[thm]{}
\newtheorem{exmp}[thm]{Example}

\newtheorem{rem}[thm]{Remark}

\theoremstyle{remark}

\begin{document}

\maketitle
\markboth{\hfill SHIHOKO ISHII \hfill}{\hfill
 GEOMETRIC PROPERTIES OF JET SCHEMES  \hfill}

\begin{abstract}
This paper shows  how properties of jet schemes relate to those of
the singularity on the base scheme. 
We will see that the jet scheme's properties of being $\bQ$-factorial, $\bQ$-Gorenstein, canonical, terminal and so on are  inherited by the base scheme.

\end{abstract}

{\footnote{partly supported by JSPS Grant-in-Aid  No.18340004 and No.19104001}}
{}

\section{Introduction}
\noindent
  When we are given a scheme $X$,  we can think of the $m$-jet scheme $X_m$ for every $m\in \bN$ associated to $X$.  
  This notion was introduced in \cite{nash} which posed the Nash problem (see 
  \cite{i-k} for one answer).
These jet schemes are objects encoding the nature of the scheme $X$ so that we can see a property of $X$ by looking at $X_m$
(see for example \cite{ein}, \cite{e-Mus}, \cite{em0}, \cite{i6}, \cite{must01}, \cite{must02}).
One typical example  is that of smoothness.
In \cite{i6}, we obtain the equivalence of the following:
\begin{enumerate}
\item
The scheme $X$ is smooth;
\item
$X_m$ is smooth for every $m\in \bN$;
\item
$X_m$ is smooth for some $m\in \bN$.
\end{enumerate}
It is natural to ask whether the same equivalences hold when we replace 
``smooth" by other properties.
In this paper we think of this problem.
We can see that  the direction $(1)\Rightarrow (2)$ or $(1)\Rightarrow (3)$ fails for many properties.
On the other hand, we can show that the implication $(3)\Rightarrow (1)$ i.e.,

``If $X_m$ has property (P) for some $m\in \bN$, then $X$ has property (P)" 

\noindent
holds for $(P)=$normal, locally complete intersection, $\bQ$-factorial, $\bQ$-Gorenstein, canonical, terminal, log-canonical, log-terminal.

We also consider the relation between a morphism of schemes $X\to Y$ and 
the correspoding morphism of jet schemes $X_m\to Y_m$.
It is proved in  \cite{i6} that a morphism $f:X\to Y$ is 
smooth (resp. unramified, \'etale) if and only if  the induced morphism 
$f_m:X_m\to Y_m$ is smooth (resp. unramified, \'etale) for some $m\in \bN$.
We consider whether the same statement holds for other properties of morphisms.
By a basic property of jet schemes, we obtain the same statement  for ``isomorphic"(see \cite{iw}).
On the other hand, about ``flat" we prove in this paper the  ``if" part and give  a counterexample for the ``only if" part.

The paper is organized as follows: In section 2, we state some basic properties on jet schemes, including the proofs for the readers' convenience.
In section 3, we show the implication:  a jet scheme $X_m$ satisfies (P)
$\Rightarrow$ the base scheme $X$ satisfies (P) for various properties (P).
In section 4, we show that if an induced morphism $f_m:X_m\to Y_m$ is flat for some $m\in \bN$, then $f:X\to Y$ is flat, but the converse does not hold.

In this paper a scheme is always defined over an algebraically closed field $k$.
In some theorems we assume  a  condition on the characteristic of $k$ according to the situation. 
When we say variety,  it means an irreducible reduced separated scheme of finite type over $k$.
Notation and terminologies on jet schemes are based on \cite{cr}.

The author would like to thank Professor Kei-ichi Watanabe for many useful comments and advices. 
She is also grateful to the members of Singularity Seminar at Nihon University for stimulating conversations and encouragements.
\section{Basics on jet schemes}

\noindent
Most of the statements in this section seem to be known by the people in the field of jet schemes.
But we give here the proofs for some of them, since we cannot find good references for them.

\begin{defn}
  Let \( X \) be a scheme of finite type over \( k \)
and $K\supset k$ a field extension.
  For \( m\in \bN \), a  \( k \)-morphism \( \sTm\to X \) is called an \( m \)-jet
  of \( X \). 
  We denote the unique point of \( \sTm \) by \( 0 \). 
\end{defn}

The space of $m$-jets (or $m$-jet scheme) of $X$, denoted by $X_m$, is a scheme of finite type over $k$,   and characterized as follows:

\begin{prop}
\label{existence}
  Let \( X \) be a scheme of finite type over \( k \).
  For every $k$-scheme $Z$ we have:
 \begin{equation}
 \label{basic}
 \Hom _{k}(Z, X_{m})\simeq\Hom _{k}(Z\times_{\spec k}
\stm, X).
\end{equation} 
 \end{prop}

In particular, for $Z=X_m$ the morphism 
$$\Lambda: X_m\times_{\spec k}
 \stm   \to  X$$ 
 corresponding to $id_{X_m}$ by the above bijection is called the {\it universal family} of $m$-jets of $X$.

For integers $m'>m\geq 0$, the canonical surjection 
$k[t]/(t^{m'+1})\to \tm$ induces the {\it truncation morphism }
$\psi_{m',m}:X_{m'} \to X_m$.
In particular for $m=0$ we denote $\psi_{m'0}$ by $\pi_{m'}$.
For every $m\in \bN$, the canonical injection $k \hookrightarrow \tm$
induces a morphism $\sigma_m:X\to X_m$.
As the injection $k \hookrightarrow \tm$
is a section of the canonical surjection $\tm\to k$,
$\sigma_m$ is a section of $\pi_{m}$.
When we should clarify the scheme $X$, we write $\psi_{m',m}^X, \pi_m^X$ and $\sigma_m ^X$.

For a morphism $f:X\to Y$ of schemes of finite type over $k$,
the induced morphism on the $m$-jet schemes is denoted by $f_m:X_m\to Y_m$.
Note that for every $\alpha \in X_m$, $f_m(\alpha)$ corresponds to the 
$m$-jet $f\circ \alpha$ of $Y$ by the bijection (\ref{basic}).

\begin{say}
   Consider \( G:=\bA^1\setminus \{0\}=\spec k[s,s^{-1}] \) as a multiplicative 
  group scheme. 
  For \( m\in \bN \), the morphism \( \tm \to 
  k[s,s^{-1}, t]/(t^{m+1}) \) defined by \( t\mapsto s t \) 
  gives an action  \[ \mu_{m}: {G} \times_{\spec k} \stm\to \stm\] of 
  \(G \) on \( \stm \).
  Therefore, it gives an action 
  \[ \mu_{Xm}:G\times_{\spec k}X_{m}\to X_{m} \]
  of \(G\) on \( X_{m} \).
\end{say}

By this action we have an $\o_X$-graded algebra  $ \oplus_{i\geq 0} \R_i$ with 
 $ \R_0=\o_X$ such that
  $$X_m=\spec \oplus_{i\geq 0} \R_i.$$

\begin{lem}
\label{cat}
  For every $m\in \bN$, the base scheme $X$ is the categorical quotient of $X_m$ by the action of $G$.
\end{lem}

\begin{proof} We refere to \cite{git} for the definition of  categorical quotient.
It is clear that there is a commutative diagram:
$$\begin{array}{ccc}
G\times X_m & \stackrel{\mu_{X_m}}\longrightarrow & X_m\\
p_2\downarrow\ \ \ \ &  &\ \ \ \downarrow \pi_m\\
X_m& \stackrel{\pi_m}\longrightarrow& X\\
\end{array}$$

If we are given a commutative diagram:
$$\begin{array}{ccc}
G\times X_m & \stackrel{\mu_{X_m}}\longrightarrow & X_m\\
p_2\downarrow\ \ \ \ &  &\ \ \ \downarrow \phi\\
X_m& \stackrel{\phi}\longrightarrow& Z,\\
\end{array}$$
let $\chi :X\to Z$ be the composite $\phi\circ \sigma_m: X\to X_m\to Z$. 
Then, we obtain the commutative diagram:
$$\begin{array}{ccc}
X_m& \stackrel{\phi}\longrightarrow & Z\\
\pi_m \downarrow\ \ \ \ & \ \ \ \ \ \nearrow  \chi & \\
X. & & \\
\end{array}$$

\end{proof}

\begin{say}
For $X=\bA^N=\spec [x_1,...,x_N]$, the $m$-jet scheme of $X$ is 
$$X_m=\spec k[x_1^{(0)},...,x_N^{(0)}, x_1^{(1)},...,x_N^{(1)},.....,x_1^{(m)},...,x_N^{(m)}]$$
with $\deg x_i^{(j)}= j$.
We  note that the universal family $\Lambda: X_m\times_{\spec k} \stm \to X$ is associated to a ring homomorphism:
$$\Lambda^*(x_i)=x_i^{(0)}+x_i^{(1)}t+\cdots +x_i^{(m)}t^m.$$
Therefore, under the ring homomorphism $\pi_m^*$  corresponding to the canonical projection $\pi_m:X_m\to X$, 
$x_i$ corresponds to $x_i^{(0)}$ for $i=1,\ldots , N$.

For a $k$-morphism $\varphi: X\to Y$ of schemes, the induced morphism
$f_m:X_m\to Y_m$ of $m$-jet schemes is $G$-equivariant.
In particular for a $k$-morphism $f:X \to Y$ of affine schemes, 
the induced ring homomorphism $f_m^*: \Gamma(Y_m, \o_{Y_m}) \to \Gamma(X_m, \o_{X_m})$  is a homomorphism of graded rings.
Therefore, for a closed subscheme $X\subset \bA^N$, the closed subscheme $X_m\subset \bA^N_m$  is defined by a homogeneous 
ideal in $k[x_1^{(0)},...,x_N^{(0)}, x_1^{(1)},...,x_N^{(1)},.....,x_1^{(m)},...,x_N^{(m)}].$

\end{say}

\begin{prop}
\label{differential}
  Let $X$ be a non-singular variety.
  Then, it follows  $$\Omega_{X_m/X_{m-1}}\simeq \pi_m^*(\Omega_{X/k}).$$
\end{prop}

\begin{proof} 
It is known that the truncation morphism $\psi_{m,m-1}:X_m\to X_{m-1}$ is a torsor by the natural action of the vector bundle \\
$\spec S(\pi_{m-1}^*\Omega_{X/k})\to X_{m-1}$, where $S(\pi_{m-1}^*\Omega_{X/k})$ is the symmetric algebra over $\o_{X_{m-1}}$ generated by $\pi_{m-1}^*\Omega_{X/k}$, where the action is defined by the addition in the fibers.
In general, a torsor $\psi: V\to Y$ with the action of the vector bundle 
$\spec S(E)\to Y$, for a locally free sheaf $E$ on $Y$,  has the relative differential 
$$\Omega_{V/Y}\simeq \psi^*E.$$
By this, in our case
we obtain that $\Omega_{X_m/X_{m-1}}\simeq \pi_m^*(\Omega_{X/k})$.
\end{proof}

\begin{cor}
\label{canonical}
  Let $X$ be a non-singular variety.
  Then, for $m\in \bN$, the canonical sheaves of $X$ and $X_m$ are related by   $$\omega_{X_m}\simeq \pi_m^*(\omega_X^{\otimes (m+1)}).$$
\end{cor}

\begin{proof}
For every $m\in \bN$, we have an exact sequence:
$$0\to \psi^*_{m,m-1}\Omega_{X_{m-1}}\to \Omega_{X_m}\to
\Omega_{X_m/X_{m-1}}\to 0.$$
By this and Proposition \ref{differential}, we obtain the isomorphism.
\end{proof}

\section{Geometric properties from jet schemes to the base scheme}

As the $k$-scheme $X$ is the categorical quotient of $X_m$ for every $m\in \bN$ by the action of $G$ (Lemma \ref{cat}), we obtain by \cite{git} the following: 
$$\begin{array}{llllll}
\mbox{(i)}&X_m&\mbox{reduced} &  \Rightarrow  &  X  &\mbox{reduced}\\
\mbox{(ii)}& X_m  &\mbox{connected} &  \Rightarrow  &  X  &\mbox{connected}\\ 
\mbox{(iii)}& X_m & \mbox{irreducible} &  \Rightarrow  &  X  &\mbox{irreducible}\\
\mbox{(iv)}& X_m & \mbox{locally\  integral}&  \Rightarrow  &  X &\mbox{locally\ integral}\\
\mbox{(v)}& X_m  &\mbox{locally\ integral} &  \Rightarrow  &  X  &\mbox{locally\ integral}\\
 && \mbox{and\ normal} &     && \mbox{and\ normal}\\
\end{array}$$

\begin{exmp}
 The converse of (i) does not hold in general.
 We give here an example in \cite{karen}. 
 Let $X$ be defined by $xy=0$ in $\bA_\bC^2$.
 Then, $X$ itself is reduced but $X_m$ is not reduced for any $m\in \bN$.
 Indeed, let  $I_m$ be the defining ideal of $X_m$ in $(\bA_\bC^2)_m$.
 Then $I_m$ is a homogeneous ideal of $ \bC[x^{(0)},y^{(0)},x^{(1)},y^{(1)},\ldots, x^{(m)}, y^{(m)}]$.
The degree 0 part of $I_m$ is generated by $x^{(0)}y^{(0)}$ and the part of degree 1 is 
generated by $x^{(0)}y^{(1)}+ x^{(1)} y^{(0)}$ as  $\bC[x^{(0)},y^{(0)}]$-modules.
Then, $f=x^{(0)}y^{(1)}\not\in I_m$, but $f^2\in I_m$. 

\end{exmp}

\begin{rem} About (ii), we have the converse statement:
If $X$ is connected, then $X_m$ is connected for every $m\in \bN$.
This can be seen as follows:
Let $P\in X_m$ be any point and let $x=\pi_m(P)$.
Then, the orbit $O_G(P)$ of $P$ by the action of $G$ is  irreducible and 
the closure $\overline{O_G(P)}$ contains $\sigma_m(x)$.
Thus, every point of $X_m$ is connected to  the section 
$\sigma_m(X)$ by an irreducible curve.
Since $\sigma_m(X)$ is connected, $X_m$ is connected.
\end{rem}

\begin{exmp}
\label{curve}
  The converse of (iii) (iv) does not hold in general.
  For example, let $X\subset \bA_\bC^3$ be a curve defined by $x^3-y^2=x^2-z^3=0$.
  Then, the main component $\overline{\pi_m^{-1}(X_{reg})}$ of $X_m$ has dimension $m+1$.
 Here, $X_{reg} $ is the open subset consisting of non-singular points of $X$.
  On the other hand, it follows $\dim \pi_m^{-1}(0)\geq m+2$ from the fact that 
  $\pi_m^{-1}(0)$ is defined by $2m-2$ equations in $(\pi_m^{\bA^3})^{-1}(0)=\bA_\bC^{3m}$.
This shows that $X_m$ is not irreducible for any $m\in \bN$.  
As $X_m$ is connected, it also shows that $X_m$ is not locally integral for $m\in \bN$.

\end{exmp}

\begin{exmp} 
\label{A1}
   The converse of (v) does not hold in general.
   For example,
   let $X$ be a normal surface defined by $x^2+y^2+z^2=0$ in $\bA_\bC^3$.
   It has an $A_1$-singularity at the origin. 
   Then, $X_m$ is irreducible by \cite{must01} but not normal for any $m\in \bN$.
   Indeed, it is known that $X_m$ is of dimension $2(m+1)$ for every $m\in \bN$.
   On the other hand, we can see that $\dim \sing(X_m)=\dim \pi_m^{-1}(0)=2m+1$, which shows that $X_m$ is not normal.

\end{exmp}

Next we will think of further properties.

\begin{thm}
\label{lci}
If $X_m$ is locally a complete intersection for an $m\in \bN$, 
then $X$ is locally a complete intersection.
\end{thm}
\begin{proof} First of all, we note that locally a complete intersection is an intrinsic property and independent of choice of smooth ambient spaces.
Since a problem is local, we may assume that $X$ is affine and $X\subset \bA^N$.
Then $X_m\subset \bA_m^N\simeq \bA^{(m+1)N}.$
Let $x\in X$ be any closed point.
Then it is sufficient to prove that if $X_m$ is a complete intersection at $P=\sigma_m(x)$, then $X$ is a complete intersection at $x$.
Let $\widehat\o_{\bA^N,x}$ and $\widehat\o_ {\bA _m^N, P}$ be $R_0$ and $R$, respectively.
Let $I\subset R_0$ and $J\subset R$
be the defining ideals of $X$ and $X_m$, respectively.
Then, $I$ and $R_0$ are direct summands of $J$ and $R$, 
respectively.
Here, the $R_0$-ideal $I$ is regarded as an $R$-ideal in the canonical way. 
By this, it follows $I\otimes_{R_0}k(x)=I\otimes_Rk(P)$.
As $I$ is a direct summand of $J$ we have an injection of $k$-vector spaces of dimension $r$ and $s$ ($r<s)$, respectively:
$$I\otimes_Rk(P)\hookrightarrow J\otimes_Rk(P).$$
By Nakayama's lemma we have elements $f_1,.., f_r\in R_0$ and $f_{r+1},.., f_s\in R$  such that
$$I=(f_1,\ldots, f_r), \ \ J=(f_1,\ldots, f_s).$$
As $X_m$ is a complete intersection at $P$, we have the equality of Krull dimensions:
$$\dim \widehat\o_{X_m,P}=\dim R/(f_1,\ldots, f_s)=N(m+1)-s,$$
since $s$ is the minimal number of generators of $J$ and $X_m$ is 
a complete intersection at $P$.
Therefore, $\dim R/(f_1,.., f_r)=N(m+1)-r$.
Since $ R/(f_1,.., f_r)=k[[x_1^{(0)},\ldots, x_N^{(m)}]]/
(f_1.., f_r)=\left(k[[x_1^{(0)},.., x_N^{(0)}]]/(f_1,.., f_r)
\right)[[x_1^{(1)},.., x_N^{(m)}]]$
we have
$$\dim R_0/I=\dim k[[x_1^{(0)},.., x_N^{(0)}]]/(f_1,.., f_r)=N-r$$
which shows that $X$ is a complete intersection at $x$.
\end{proof}

\begin{exmp}
If $X$ is  locally a complete intersection, then $X_m$ is not necessarily  locally a complete intersection.
Example \ref{curve} shows such an example. 
\end{exmp}

\begin{thm}
\label{Qfac}
If $X_m$ is a $\bQ$-factorial variety for an $m\in \bN$, 
then $X$ is $\bQ$-factorial.
\end{thm}

\begin{proof} As $X_m$ is a $\bQ$-factorial variety, $X_m$ is normal and integral, therefore $X$ is normal and integral by the statement at the beginning of this section.
It is sufficient to prove that for a prime divisor $D\subset X$ there is an integer $r>0$ such that $rD$ is a Cartier divisor.
Since $X$ is normal, $D\cap X_{reg}\neq \emptyset$.
Let $\tilde D=\overline{\pi_m^{-1}(D\cap X_{reg})}\subset X_m$,
then by the assumption of the theorem there is an integer $r>0$ such that $r\tilde D$ is a Cartier divisor.
Then, the subscheme $\tilde D\cap \sigma_m(X)$ is isomoprhic to $D$ by the isomorphism $\pi_m|_{\sigma_m(X)}:\sigma_m(X)\simeq X$.
This follows from the following diagram, where all inclusions become equalities:
$$\begin{array}{ccccc}
\pi_m^{-1}(D)\cap \sigma_m(X)&=&\left(\pi_m|_{\sigma_m(X)}\right)^{-1}(D)&
\simeq & D\\
\bigcup&&&&\| \\
\tilde D\cap \sigma_m(X) & \supset &\overline{\left(\pi_m|_{\sigma_m(X)}\right)^{-1}(D\cap X_{reg})} & 
\simeq & D.\\
\end{array}$$
Therefore we have only to show that $\tilde D\cap \sigma_m(X)$ is a $\bQ$-Cartier divisor on $\sigma_m(X)$.
Now we have 
\begin{equation}\label{divisor}
(r\tilde D)\cap \sigma_m(X)=r(\tilde D\cap \sigma_m(X)),
\end{equation}
which follows from the fact that the both divisor coincide on an open subset $\sigma_m(X)_{reg}$ by 
$$(r\tilde D)\cap \sigma_m(X)_{reg}=((r\tilde D)\cap (X_{reg})_m)\cap \sigma_m(X)$$
$$=r(\tilde D\cap (X_{reg})_m)\cap\sigma_m(X) = r(\tilde D\cap \sigma_m(X)_{reg}),$$
where the last equality follows from the fact that $\tilde D\cap (X_{reg})_m$ is a Cartier divisor.
As $(r\tilde D)\cap \sigma_m(X)$ in (\ref{divisor}) is a Cartier divisor, it follows that 
$ r(\tilde D\cap \sigma_m(X))$ is a Cartier divisor on $\sigma_m(X)$.
\end{proof}

When we say that a scheme is $\bQ$-factorial, $\bQ$-Gorenstein, canonical, log-canonical, terminal, log-terminal, we always assume that the scheme is integral and normal.
Let any of these property be (P), then 

``if $X$ satisfies (P), then $X_m$ satisfies (P) "
\newline 
does not hold as we already have Example \ref{A1} that $X$ satisfies (P)  but $X_m$ is not even normal for any $m\in \bN$.

\begin{thm}
\label{Qgor}
If $X_m$ is $\bQ$-Gorenstein of index $r$ for an $m\in \bN$, 
then $X$ is $\bQ$-Gorenstein of index $\leq r(m+1)$.
\end{thm}

\begin{proof}
We may assume that $X$ is a normal affine variety.
By Proposition \ref{canonical}, 
$$K_{(X_{reg})_m}=(m+1)\pi_m^*(K_{X_{reg}}).$$
Take an effective divisor $D_0\in |r(m+1)K_X|$ and let 
$D=\overline{\pi_m^*(D_0|_{X_{reg}})}$.
Then,
$$rK_{X_m}=D+F, \ \ \  \supp F\cap (X_{reg})_m=\emptyset.$$
Since the Cartier divisor $D+F$ on $X_m$ does not contain the section
$\sigma_m(X)$ in its support, the restriction $(D+F)|_{\sigma_m(X)}$ can be defined and it is also a Cartier divisor.
By the definition of $F$, we have 
$$(D+F)|_{\sigma_m(X)_{reg}}=D|_{\sigma_m(X)_{reg}}.$$
On the other hand, we have that $D|_{\sigma_m(X)_{reg}}$ corresponds to $ D_0|_{X_{reg}}$
by the isomorphism $\sigma_m(X)\simeq X$.
Therefore, we have 
$(D+F)|_{\sigma_m(X)}$ correponds to $ D_0$
which shows that $D_0$ is a Cartier divisor.

\end{proof}

\begin{lem}
\label{diagram}
Assume the characteristic of $k$ is zero.
Let $f:Y\to X$ be a resolution of  the singularities of $X$.
For  $m\in \bN$, there is a commutative diagram:

\begin{equation}
\label{d}
\xymatrix{ \tilde {Y} \ar[rdd]_{\varphi} \ar[rrd]^{\psi}& & \\
& Y_m  \ar@{^{(}->}[lu] \ar[r]^{f_m} \ar[d]^{\pi_m^Y} & X_m \ar[d]^{\pi_m^X}\\
&Y \ar[r]_f & X\\
 }
\end{equation}
  such that $\tilde Y$ is nonsingular and containing $Y_m$ as an open subset and $\psi$ is proper and birational.
\end{lem}
\begin{proof}
  Let $Y_m=\spec \R$, where $\R$ is a graded $\o_Y$-algebra and 
  $\R=\oplus_{i\geq 0}\R_i$ the homogeneous decomposition.
  Let $Y'=\Proj \oplus_{i\geq 0}S_iT^i $ with $\deg T=1$ and
  $S_i=\R_i\oplus \R_{i-1}\cdots \oplus \o_Y$, then $Y_m$ is canonically contained in  $Y'$ as an open subscheme and there is a proper morphism $\varphi': Y'\to Y$. 
  Let $Y''=\overline{Im (id_{Y_m}, f_m)}\subset Y'\times_X X_m$ and let $g:\tilde Y\to Y''$ 
  be a resolution of the singularities of $Y''$.
Let $\varphi$ be the composite $\tilde Y\to Y''\to Y' \to Y $ and
 $\psi$ be the composite $\tilde Y \to Y''\to X_m$.
 Then the birationality and properness of $\psi$ follows from the birationality of $f_m$ and the properness of 
 $f$, respectively.
\end{proof}

\begin{thm}
\label{cano}
 Assume $char k=0$.
If $X_m$ has at worst canonical (resp. terminal, log-terminal) singularities for an $m\in \bN$, 
then $X$ has at worst canonical (resp. terminal, log-terminal) singularities.
\end{thm}

\begin{proof}
  By Theorem \ref{Qgor},  $X$ is $\bQ$-Gorenstein.
  By Lemma \ref{diagram}, we obtain a proper birational morphisim 
  $\psi:\tilde Y \to X_m$ with non-singular $\tilde Y$ such that the diagram 
  (\ref{d}) is commutative.
  Then, it follows:
  $$K_{\tilde Y}=\psi^*K_{X_m}+\sum_i a_i F_i,$$
  for $\psi$-exceptional prime divisors  $F_i$.
  As an exceptional prime divisor  for $f_m$ is the restriction $F_i|_{Y_m}$ and it is of  the form $(\pi_m^Y)^*(E_i)$ for a prime divisor $E_i$ on $Y$, we obtain
  $$K_{Y_m}=f_m^*K_{X_m}+\sum_i a_i (\pi_m^Y)^*(E_i).$$
 Restricting this equality on the section $\sigma_m(Y)$,
 we obtain 
 \begin{equation}
\label{hikaku}
 K_{Y_m}|_{\sigma_m(Y)}=f_m^*(K_{X_m}|_{\sigma_m(X)})+\sum_i a_i (\pi_m^Y)^*(E_i)|_{\sigma_m(Y)}.
 \end{equation}
 As we have  $K_{X_m}|_{\sigma_m(X)}=(m+1)K_X$ with identifying $X$ and $\sigma_m(X)$, the equality (\ref{hikaku}) is transfered by the isomorphism $\sigma_m(Y)\simeq Y$ as follows:
\begin{equation}
\label{coef}
(m+1)K_Y=f^*((m+1)K_X)+ \sum_i a_i E_i.
\end{equation}
 Therefore, $a_i\geq 0$ (resp. $a_i> 0$, $a_i>-1$) for every $i$ implies $X$ has at worst canonical (resp. terminal, log-terminal) singularities.
\end{proof}

\begin{thm}
\label{logcano}
Assume $char k=0$.
If $X_m$ has at worst log-canonical singularities for an $m\in \bN$, 
then $X$ has at worst  log-terminal singularities.
\end{thm}

\begin{proof}
  Look at the equality (\ref{coef}) in the proof of Theorem \ref{cano}.
  The log-canonicity condition $a_i\geq -1$ of $X_m$ implies 
  $$\frac{a_i}{m+1}> -1$$ which is the log-terminality condition of $X$.
 \end{proof}

\begin{exmp}
Let $X$ be a hypersurface of $\bA_\bC^n$ defined by $x_1^2+x_2^2+\cdots+
x_n^2=0$. It is well known that $X$ has an  isolated singularity at the origin 0,
which is a canonical singularity for $n\geq 3$ and a terminal singularity for $n\geq 4$. 
We will show that $X_1$ has canonical singularities for $n\geq 4$, and has
terminal singularities for $n\geq 5$.
The 1-jet scheme $X_1$ is defined by
$$\sum_{i=1}^nx_i^2=0,\ \  \sum_{i=1}^n x_ix^{(1)}_i=0$$
in $\bA^{2n}=\spec \bC[x_1,\ldots, x_n,x_1^{(1)},\ldots, x_n^{(1)}]$.
Let $H$ and $L$ be the hypersurfaces in $\bA^{2n}$ defined by $\sum_{i=1}^nx_i^2=0$ and
$\sum x_ix^{(1)}_i=0$, respectively.
Let $f:Z\to \bA^{2n}$  be the blow up of $\bA^{2n}$ by the center $\sing X_1=\{x_1=\cdots=x_n=0\}$.
Then the proper transform $\tilde X_1$ of $X_1$ is the intersection $\tilde H\cap \tilde L$ of the proper transforms $\tilde H$ and $\tilde L$ of $H$ and $L$,
respectively.
We can see that $\tilde X_1$ is non-singular, therefore $\tilde X_1$ is a resolution of the singularities of $X_1$.
Let $E$ be the exceptional divisor of $f$, then $\tilde H\sim -2E$ and $\tilde L
\sim -E$.
By this, it follows
$$K_{\tilde H}=(K_Z+\tilde H)|_{\tilde H}\sim((n-1)E-2E)|_{\tilde H}=(n-3)E|_{\tilde H}$$
$$K_{\tilde H\cap\tilde L}=(K_{\tilde H}+\tilde L)|_{\tilde H\cap\tilde L}\sim((n-3)E-E)|_{\tilde H\cap\tilde L}=(n-4)E|_{\tilde{X_1}}.$$
Therefore, $n\geq 4$ if and only if $X_1$ is canonical and $n>4$ if and only if $X_1$ is terminal.

\end{exmp}

\section{Morphisms of jet schemes}
 \begin{thm}
   Let $f:X \to Y$ be a morphism of $k$-schemes.
   If the induced morphism $f_m:X_m\to Y_m$ is flat for some $m\in \bN$,
   then $f$ is flat.
 \end{thm}
 
 \begin{proof}
 We will show that $f_m:X_m\to Y_m$ is not flat for every $m\in \bN$ if 
 $f:X \to Y$ is not flat.
 As the problem is local, we may assume that $X$ and $Y$ are affine schemes $\spec A$ and $\spec B$, respectively.
 Then, the $m$-jet schemes $X_m$ and $Y_m$ are also affine. 
 Denote $X_m=\spec \overline{A}$ and $Y_m=\spec \overline{B}$.
 Here, $\oa$ is a graded $A$-algebra with $\oa_0=A$ and $\ob$ is a graded $B$-algebra with $\ob_0=B$.
 Decompose $\oa,\ob$ as follows:
 $$\oa=A\oplus \oa_+,\ \ \ \ob=B\oplus \ob_+,$$
 where $\oa_+,  \ob_+$ are the parts of positive degree.
 We note that every $B$-module $M$ can be regarded as a $\ob$-module in the canonical way, i.e. for $b=b_0+b_+\in \ob$ $(b_0\in B, b_+\in \ob_+)$ and $x\in M$, define $bx$ by 
 $b_0x$.
 By this, we have $M\otimes_BA=M\otimes_{\ob} A$ for every $B$-module $M$.
 
 For a proof of the theorem, we assume that $f^*:B\to A$ is not flat.
 By \cite[Theorem 77]{mats}, it is equivalent to that there exists an ideal $I\subset B$ such that the canonical map 
 $$I\otimes_BA\stackrel{\psi}\longrightarrow A$$ 
 is not injective.
 This map appears in the following commutative diagram:
 $$
 \begin{array}{ccc}
 I\otimes_BA&\stackrel{\psi}\longrightarrow& A\\
 \varphi \downarrow\ \ \ && \bigcap\\
 I\ob\otimes_{\ob}\oa&\stackrel{\overline\psi}\longrightarrow &\oa.\\
 \end{array}$$
 Here, by the $\ob$-modules decomposition $I\ob=I\oplus I\ob_+$, we have $ I\ob\otimes_{\ob}\oa=(I\oplus I\ob_+)
\otimes_{\ob}(A\oplus \oa_+)$, which shows that $I\otimes_{\ob}A=I\otimes_BA$ is a direct summand of  $ I\ob\otimes_{\ob}\oa$.
Therefore the map $\varphi$ is injective and the non-injectivity of $\psi$ yields the non-injectivity of $\overline\psi$, which shows the non-flatness of 
$\ob\to \oa$.
 \end{proof}
 
\begin{exmp}
The converse of the theorem does not hold.
 Let $X\subset \bA_\bC^3$ be defined  by the equation
 $t^d+x^d+y^d=0, $ with $d\geq 3$, then it is  a normal surface with the singularity at the  origin $\bold 0=(0,0,0)$.
 Let $f:X\to Y=\bA_\bC^1$ be the first projection $(t,x,y)\mapsto t$.
 Then, as $f$ is a surjective morphism from a reduced scheme to a non-singular curve, it is flat.
 However, for every $m\geq 2$ the induced morphism $f_m:X_m\to Y_m$ is non-flat.
 This is shown as follows: 
 For every $m\in \bN$, consider the commutative diagram:
 $$\begin{array}{ccc}
 X_m&\stackrel{f_m}\longrightarrow& Y_m\\
\pi_m^X \downarrow\ \ \ \ & & \ \ \ \ \downarrow\pi_m^Y\\
 X&\stackrel{f}\longrightarrow& Y\\
 \end{array}
$$
As $\pi_m^Y$ is smooth, it is sufficient to prove that $f\circ \pi_m^X$ is not flat for $m\geq 2$.
Note that
$\pi_m^{-1}(X\setminus \{\bold 0\})$ is irreducible and of dimension 
$2(m+1)$.

For $m<d$, 
$(\pi_m^X)^{-1}(\bold 0)=(\pi_m^{\bA^3})^{-1}(\bold 0)=\bA^{3m}$.
For $m\geq d$, as $(\pi_m^X)^{-1}(\bold 0)$ is defined by $m+1-d$ equations in $\bA^{3m}$, it follows 
$$\dim (\pi_m^X)^{-1}(\bold 0)\geq 3m-(m+1)+d\geq 2(m+1).$$
If we assume that $m\geq 2$, in the both cases above we have 
$$\dim (f\circ \pi_m^X)^{-1}(0)\geq \dim (\pi_m^X)^{-1}(\bold 0)>2m+1=
\dim (f\circ \pi_m^X)^{-1}(t),$$
where $0\neq t\in Y$.
This yields that $f\circ \pi_m^X$ is not flat.

\end{exmp}

\makeatletter \renewcommand{\@biblabel}[1]{\hfill#1.}\makeatother


\begin{thebibliography}{11}

\bibitem{ein} L. Ein, M. Musta\cedilla{t}\v{a} and T. Yasuda,
{\em Jet schemes, log discrepancies and inversion of adjunction,}
 Invent. Math. {\bf 153} (2003) 519-535.

\bibitem{e-Mus} L. Ein and M. Musta\cedilla{t}\v{a}. {\em Inversion of 
Adjunction for local complete intersection varieties}, Amer. J. 
Math. 126 (2004), 1355--1365.


\bibitem{em0} L. Ein and M. Musta\cedilla{t}\v{a}. {\em Invariants of singularities of pairs}, preprint, arXiv: math.AG/0604601, to appear in Proceedings of the International Congress of Mathematicians, Madrid, Spain, 2006.

\bibitem{em} L. Ein and M. Musta\cedilla{t}\v{a}, {\em Jet schemes and singularities}, preprint,
arXiv:math.AG/0612862.

\bibitem{karen} R.A. Goward, Jr and K. Smith, {\em The jet scheme of a monomial scheme}, Com. Alg. {\bf 34} (2006), 1591--1598.


\bibitem{ha} R. Hartshorne, {\em Algebraic Geormetry}, 
Graduate Texts in Math. {\bf 52} Springer-Verlag, (1977).


\bibitem{i-k} S. Ishii and J. Koll\'ar, {\em The Nash problem   on  
arc families of singularities,}  Duke 
Math. J. {\bf 120} No.3 (2003) 601-620.


\bibitem{cr} S. Ishii, {\em Jet schemes, arc spaces and the Nash problem}, C.R. Math. Rep. Acad. Sci. Canada {\bf 29} (1) (2007) 1--21.

\bibitem{i6} S. Ishii, {\em Smoothness and jet schemes}, Adv. St. Pure Math.
{\bf 56},  Proceedings of ``Singularities in Niigata$\cdot$Toyama 2007" (2009) 187--199.


\bibitem{iw} S. Ishii and J. Winkelmann, {Isomorphisms of jet schemes}, 
preprint arXiv:0908.1146, 
to appear in C.R. Math. Rep. Acad. Sci. Canada.


\bibitem{mats} H. Matsumura, {\em Commutative Ring Theory}, Cambridge Studies in Adv. Math. {\bf 8} (1986).


\bibitem{git} D. Mumford, J. Fogarty, F. Kirwan, {\em Geometric Invariant Theory}, Ergebnisse der Mathematik und ihrer Grenzgebiete 34, Springer-Verlag (1994).


\bibitem{must01} M. Musta\cedilla{t}\v{a}, {\em Jet schemes of locally 
complete intersection canonical singularities,} with an appendix by 
David Eisenbud and Edward Frenkel, Invent. Math. {\bf 145} (2001) 
397--424.

\bibitem{must02} M. Musta\cedilla{t}\v{a}, {\em
Singularities of Pairs via Jet Schemes},
 J. Amer. Math. Soc. {\bf 15} (2002), 599--615.
 
\bibitem{nash} J. F. Nash, {\em Arc structure of singularities},
Duke Math. J. {\bf 81}, (1995) 31--38. 



\end{thebibliography}
\end{document}